\documentclass[12pt,centertags,oneside]{amsart}
\usepackage{amsmath,amstext,amsthm,amscd,typearea,hyperref}
\usepackage{amssymb}
\usepackage{a4wide}
\usepackage[mathscr]{eucal}
\usepackage{mathrsfs}
\usepackage{typearea}
\usepackage{charter}
\usepackage{pdfsync}
\usepackage{stmaryrd}
\usepackage[T1]{fontenc}
\usepackage{enumitem}
\usepackage{dsfont}
\usepackage{verbatim}

\usepackage{tikz-cd}
\usepackage{pgfplots}
\tikzset{>=latex}
\pgfplotsset{compat=newest}
\usepackage{caption}
\usepackage{adjustbox}
\usepackage[a4paper,width=16.5cm,top=3cm,bottom=3cm]{geometry}

\numberwithin{equation}{section}

\usepackage[french,english]{babel}

\allowdisplaybreaks
\emergencystretch=\maxdimen
\usepackage{multicol}

\usepackage{xcolor}

\newtheorem{theorem}{Theorem}[section]
\newtheorem{definition}[theorem]{Definition}
\newtheorem{proposition}[theorem]{Proposition}
\newtheorem{corollary}[theorem]{Corollary}
\newtheorem{lemma}[theorem]{Lemma}
\newtheorem{remark}[theorem]{Remark}

\newtheorem*{definition*}{Definition}

\newcommand{\cali}[1]{\mathscr{#1}}

\newcommand{\supp}{{\rm supp}}

\newcommand{\diff}{{\rm d}}

\newcommand{\DSH}{{\rm DSH}}

\newcommand{\del}{\partial}

\newcommand{\dist}{\mathop{\mathrm{dist}}\nolimits}

\newcommand{\ddc}{{\rm dd^c}}
\newcommand{\dc}{{\rm d^c}}

\newcommand{\dd}{{\rm d}}

\def\d{\operatorname{d}}

\newcommand{\dbar}{\overline\partial}
\newcommand{\ddbar}{\partial\overline\partial}

\newcommand{\vep}{\varepsilon}

\newcommand{\Leb}{{\rm Leb}}

\newcommand{\qpsh}{{\rm qpsh}}

\newcommand{\Cc}{\cali{C}}
\newcommand{\Dc}{\cali{D}}

\newcommand{\Fc}{\cali{F}}

\newcommand{\Uc}{\cali{U}}

\newcommand{\cA}{\mathcal{A}}
\newcommand{\cB}{\mathcal{B}}

\newcommand{\D}{\mathbb{D}}
\newcommand{\C}{\mathbb{C}}

\newcommand{\N}{\mathbb{N}}

\newcommand{\R}{\mathbb{R}}

\renewcommand\P{\mathbb{P}}

\newcommand\restrict[1]{\raisebox{-.5ex}{$|$}_{#1}}

\newcommand{\lp}{\langle}
\newcommand{\rp}{\rangle}

\newcommand{\norm}[1]{\lVert#1\rVert}

\newcommand{\oA}{\mathcal{A}}
\newcommand{\oB}{\mathcal{B}}

\newcommand{\oN}{\mathcal{N}}

\newcommand{\oE}{\mathcal{E}}

\newcommand{\bv}{\mathbf{v}}


\title[Exponential mixing of all orders on K\"ahler manifolds]{Exponential mixing of all orders on K\"ahler manifolds: (quasi-)plurisubharmonic observables}

\author{Marco Vergamini}
\address{Scuola Normale Superiore, Pisa, Italy}
\email{marco.vergamini@sns.it}

\author{Hao Wu}
\address{Nanjing University, Nanjing, China}
\email{haowu@nju.edu.cn}

\thanks{}

\begin{document}

\begin{abstract}
Let $f$ be a holomorphic automorphism of a compact K\"ahler manifold with simple action on cohomology and $\mu$ its unique measure of maximal entropy. We prove that $\mu$ is exponentially mixing of all orders for all d.s.h.\ observables, i.e., functions that are locally differences of plurisubharmonic functions. As a consequence, every d.s.h.\ observable satisfies the central limit theorem with respect to $\mu$.
\end{abstract}

\clearpage\maketitle
\thispagestyle{empty}


\medskip


\setcounter{tocdepth}{1}
%\tableofcontent

\section{Introduction} \label{intro}

Let $(X, \omega)$ be a compact K\"ahler manifold of dimension $k$ and $f$ a holomorphic automorphism of $X$. We refer to \cite{C01,D-TD12,DS05} for the general properties of such maps. We denote by $f^n$ the $n$-th iterate of $f$. For $0\le q\le k$, the \emph{dynamical degree of order $q$} of $f$ is the spectral radius of the pull-back operator $f^*$ acting on the Hodge cohomology group $H^{q,q}(X,\mathbb{R})$. It is denoted by $d_q(f)$, or simply by $d_q$ if there is no confusion. By Poincaré duality, the dynamical degree $d_q$ of $f$ is equal to the dynamical degree $d_{k-q}(f^{-1})$ of $f^{-1}$. We have $d_0=d_k=1$ and $d_q(f^n)=d_q^n$ for all $q$.

A theorem by Khovanskii \cite{K79}, Teissier \cite{Te79}, and Gromov \cite{G90} implies that the sequence $q\mapsto\log{d_q}$ is concave. So, there are integers $0\le p\le p'\le k$ such that
$$1=d_0<\dots<d_p=\dots=d_{p'}>\dots>d_k=1.$$

We assume that $f$ has \emph{simple action on cohomology}, i.e., that we have $p=p'$ and $f^*$, acting on $H^{p,p}(X,\mathbb{R})$, admits only one eigenvalue of maximal modulus $d_p$. We fix a constant $\max\{d_{p-1},d_{p+1}\}<\delta_0<d_p$ such that all the eigenvalues of $f^*$ acting on $H^{p,p}(X,\mathbb{R})$, except for $d_p$, have modulus smaller than $\delta_0$. We call $d_p$ the \emph{main dynamical degree} and $\delta_0$ the \emph{auxiliary dynamical degree} of $f$.

From \cite{C01, DS05, ds-superpot-kahler} we know that $f$ admits a unique probability measure of maximal entropy $\mu$, called the \emph{equilibrium measure} of $f$, which is the intersection of a positive closed $(p, p)$-current $T_+$ and a positive closed $(k-p, k-p)$-current $T_-$ (the \emph{Green currents} of $f$ and $f^{-1}$, respectively). A main question in the domain is to study the statistical properties of $\mu$. The major difficulties in this setting are the presence of both attractive and repelling directions and the non uniform hyperbolicity of the system. The goal of this paper is to address these questions for a large class of natural observables.

\smallskip

The simplest holomorphic dynamical systems displaying both the difficulties above are given by complex H\'enon maps, see, e.g., \cite{bedford1,bed-1991,forn}. In this case, the exponential mixing for two H\"older-continuous observables was first established by Dinh in \cite{dinh-decay-henon}. It was recently extended by Bianchi-Dinh in \cite{bian-dinh-sigma} to any number of observables, and by the authors in \cite{wu-vergamini-24, Wu-Ergodic} to all plurisubharmonic (p.s.h.) observables. We also refer to \cite{hen-gab-CLT,vigny-decay} for the case of generic birational maps of $\P^k$ and to \cite{bd-gafa,dinh-exponential} for the case of holomorphic endomorphisms of $\P^k$.

\smallskip

On a compact K\"ahler manifold, p.s.h.\ functions are constant. So we consider in this paper \emph{d.s.h.}\ observables, which are, roughly speaking, locally differences of p.s.h.\ functions, see \cite{din-sib-cmh} and Subsection \ref{dsh-subsec} for the precise definition. The following is our main result, which settles the problem of mixing for d.s.h.\ observables on compact K\"ahler manifolds.

\begin{theorem} \label{thm-mixing}
Let $f$ be a holomorphic automorphism of a compact K\"ahler manifold $(X, \omega)$ of dimension $k$. Assume that $f$ has simple action on cohomology, let $\mu$ be its equilibrium measure and $d_p$ be its main dynamical degree. Then, $\mu$ is exponentially mixing of all orders for all observables in $\DSH(X)$. More precisely, there exists $0<\delta'<d_p$ such that for every $\delta'<\delta<d_p$, every integers $\kappa\in \N^*$,  $0=n_0\leq n_1\leq \ldots\leq n_\kappa$ and every $\varphi_0,\varphi_1,\dots, \varphi_\kappa \in \DSH(X)$, we have
$$\Big| \int \varphi_0 (\varphi_1\circ f^{n_1})\cdots (  \varphi_\kappa \circ f^{n_\kappa}) \,\dd \mu -\prod_{j=0}^\kappa  \int \varphi_j \,\dd \mu \Big| \leq C_{\delta,\kappa}  \Big(\frac{\delta}{d_p}\Big) ^{\min_{0\leq j\leq \kappa-1}(n_{j+1}-n_j)/2}\prod_{j=0}^\kappa \norm{\varphi_j}_\DSH,$$
where  $C_{\delta,\kappa}>0$ is a constant independent of $n_1,\dots,n_\kappa,\varphi_0,\dots,\varphi_\kappa$.
\end{theorem}

We refer to \cite{bd-kahler,DS10CM} for the more regular case of $\Cc^2$-continuous observables and to \cite{wu-kahler} for the case $\kappa=1$. Observe that all d.s.h.\ functions are in $L^r(\mu)$ for every $r\ge1$ \cite{dinh-exponential}, hence all the integrals above are well defined.

\smallskip

Our proof in \cite{wu-vergamini-24} for the case of H\'enon maps relies on precise estimates for p.s.h.\ functions and on the homogeneous structure of $\P^2$. As non-trivial p.s.h.\ functions do not exist on compact K\"ahler manifolds, both these ingredients are not available now. Instead, we will make a crucial use of the theory of \emph{super-potentials} by Dinh-Sibony \cite{dinh-sibony:acta,ds-superpot-kahler}, which permits to quantify the regularity of currents of arbitrary degree when seen as operators on appropriate spaces of forms.

\medskip

A consequence of our main theorem is that all d.s.h.\ observables satisfy the central limit theorem. More precisely, fix an observable $\varphi\in \DSH(X)$ and set $S_n(\varphi):=  \varphi +\varphi \circ f +\cdots +\varphi \circ f^{n-1}$. By Birkhoff’s ergodic theorem, we have $n^{-1} S_n (\varphi) (x) \to \lp \mu,\varphi \rp$ for $\mu$-almost every $x\in X$.
As in \cite{wu-vergamini-24}, the following control of the rate of the convergence is a consequence of Theorem \ref{thm-mixing} and \cite[Theorem 4.1]{wu-vergamini-24}, which is an adapted version of the criterion in \cite{bjo-gor-clt}. We let $\oN(0,\sigma^2)$ denote the Gaussian distribution with mean $0$ and variance $\sigma^2$ (when $\sigma=0$, we mean that $\oN(0,\sigma^2)$ is the trivial point distribution at $0$).

\begin{corollary}\label{thm-clt}
Let $X,f$ and $\mu$ be as in Theorem \ref{thm-mixing}. Then, every $\varphi\in\DSH(X)$ satisfies the central limit theorem with respect to $\mu$. Namely, we have
\begin{equation} \label{law-conv}
    \frac{S_n(\varphi )-  n\lp \mu, \varphi \rp}{\sqrt n} \longrightarrow  \oN(0,\sigma^2)  \quad \text{as }\, n\to\infty \quad \text{in law},
\end{equation}
where  $$\sigma^2 := \lim_{n\to \infty} \frac 1n\int \big(S_n(\varphi) -  \lp \mu, \varphi\rp \big)^2 \,\dd \mu.$$
\end{corollary}

\medskip

\noindent\textbf{Notations.}
The symbols $\lesssim$ and $\gtrsim$ stand for inequalities up to a positive multiplicative constant, and a subscript means that said constant can depend on some variables, e.g., $\lesssim_t$ means that the implicit constant can depend on the variable $t$. The pairing $\lp \cdot,\cdot \rp$ is used for the integral of a function with respect to a measure or, more generally, the value of a current at a test form. The mass of a positive closed current $S$ of bidegree $(q,q)$ on a compact K\"ahler manifold $(X,\omega)$ of dimension $k$ is defined as $\|S\|:=\lp S,\omega^{k-q}\rp$. If $U$ is an open set in $\C^k$, we denote by $bU$ the topological boundary of $U$, i.e., $bU:=\overline{U}\setminus U$.

\medskip

\noindent\textbf{Acknowledgements.}
The first author is part of the PHC Galileo project G24-123. He would also like to thank the National University of Singapore for the warm welcome and the excellent work conditions.

\section{Preliminaries} \label{defs}

\subsection{Quasi-plurisubharmonic and d.s.h.\ functions} \label{dsh-subsec}

We fix in this section a compact K\"ahler manifold $(X,\omega)$. A function $\varphi:X\rightarrow\R \cup \{-\infty\}$ is called \emph{quasi-plurisubharmonic}  (\emph{quasi-p.s.h.}\ for short) if, locally, it is the difference of a p.s.h.\ function and a smooth one. A function $\varphi:X\rightarrow\R \cup\{\pm \infty\}$ is \emph{d.s.h.}\ \cite{din-sib-cmh,dinh-sibony:cime} if it is the difference of two quasi-p.s.h.\ functions outside of a pluripolar set. Denote by $\DSH(X)$ the space of d.s.h.\ functions on $X$. If $\varphi$ is d.s.h., there are two positive closed $(1,1)$-currents $R^\pm$ on $X$ such that $\ddc \varphi =R^+ -R^-$. As these two currents are cohomologous, they have the same mass. We define a norm on $\DSH(X)$ by 
$$\norm{\varphi}_\DSH:= \Big| \int  \varphi \,\omega^k  \Big|+\inf \norm {R^\pm},$$
where the infimum is taken over all $R^\pm$ as above. We obtain an equivalent norm if, instead of $\omega^k$, we take any measure $\nu$ that is \emph{PB}, i.e., such that all d.s.h.\ functions are integrable with respect to $\nu$.
We will need the following decomposition result for d.s.h.\ functions, see for instance \cite{din-sib-cmh} and \cite[Lemma 2.1]{wu-vergamini-24}.

\begin{lemma} \label{dsh-split-psh}
Let $\varphi$ be a d.s.h.\ function on $X$ with $\|\varphi\|_\DSH\le1$. There exist two functions $\varphi_+$ and $\varphi_-$ which are quasi-p.s.h.\ and such that
$$\ddc\varphi_\pm\ge-C\omega,\qquad\|\varphi_\pm\|_\DSH\le C,\qquad\varphi_\pm\le0,\qquad\text{and}\qquad\varphi=\varphi_+-\varphi_-,$$
where $C$ is a positive constant that depends on $(X,\omega)$ but is independent of $\varphi$.
\end{lemma}

 Let $\rho(z):=\widetilde\rho(|z|)$ be a radial function  on $\C^k$ such that 
    $$\widetilde\rho\geq 0, \qquad \widetilde\rho(t)=0 \,\text{ for }\, t\geq 1, \qquad \text{and}\qquad \int_{\C^k} \rho\,\dd \Leb =1.$$
    For $\vep>0$, we set $\rho_\vep (z):=\vep^{-2k}\rho(z/\vep)$. For every function $u$ on an open set $U\subset \C^k$ and every subset $U'\Subset U$, define
    \begin{equation} \label{convolution}
        u_\vep(z):=(u* \rho_\vep)(z)=\int_{|w|\leq 1}u(z-\vep w)\rho(w) \,\dd\Leb(w)  \quad\text{for}\quad z\in U'
    \end{equation}
    provided that $0<\vep<\dist(U',bU)$.

\begin{lemma} \label{lem-u-vep}
Let $U'\Subset U$ be open subsets of $\C^k$
and $u$ a bounded p.s.h.\ function on $U$. For every  $0<\vep<\dist(U',bU)$, we have
    \begin{equation*}
\|u_\varepsilon-u\|_{L^1(U',\Leb)} \lesssim_{U,U'} -\|u\|_{L^\infty(U)}\varepsilon\log{\varepsilon}.
    \end{equation*}
 \end{lemma}   
    \proof
    The proof uses standard arguments, but we give it for the reader's convenience. We will proceed in three steps.

   \smallskip
    \noindent
{\bf Step 1.} For every compact set $K\subseteq \C$ and every finite positive measure $\nu$ on $\C$ whose support is compactly contained in a ball $B$ of radius $R$ containing $K$, we have that
\begin{equation} \label{convolution-estimate-2}
    \int_K \big|u_\nu(z-w)-u_\nu(z)\big|\,\diff\Leb(z) \lesssim_{K,R}-\nu(B)\varepsilon\log{\varepsilon}\quad\text{for every}\quad  |w|\le\varepsilon,
\end{equation}
where $u_\nu(z):= \int_\C \log|z-\zeta|\,\diff\nu(\zeta)$.

\medskip

\textit{Proof of Step 1.} We have the following estimate:
$$\int_K \big|\log|z-\varepsilon|-\log|z|\big| \,\diff\Leb(z) \lesssim_{K} -\varepsilon\log{\varepsilon}.$$
Combining it with the definition of $u_\nu$, we get \eqref{convolution-estimate-2}.
    
\medskip    
    \noindent
{\bf Step 2.} For every open set $V\subseteq \C$, every compact set $K\Subset V$, and every function $u$ which is subharmonic and bounded in $V$, we have that
\begin{equation} \label{convestquasifinal}
    \int_K \big|u(z-w)-u(z)\big|\,\diff\Leb(z) \lesssim_{K,V}-\|u\|_{L^\infty(V)}\varepsilon\log{\varepsilon}.
\end{equation}
\textit{Proof of Step 2.} Assume without loss of generality that $\|u\|_{L^\infty(V)}=1$. Let $K_\eta$ be the $\eta$-neighborhood of $K$. Choose $\eta$ sufficiently small to have $K_{3\eta}\Subset V$, and take $\chi_\eta$ a positive smooth cut-off function with $\chi_\eta\restrict{K_{2\eta}}\equiv1$ and $\supp{\chi_\eta}\Subset K_{3\eta}$. Define $\nu$ to be equal to $\chi_\eta\cdot\ddc u$ on $V$ and to $0$ outside of $V$. We have $\nu(B) \lesssim_{K,\eta,V} \|u\|_{L^\infty(V)}$, where $B$ is a large ball containing $K_{3\eta}$. Consider $u_\nu$ defined as in Step 1. Since $\nu$ satisfies the hypothesis of Step 1, $u_\nu$ satisfies inequality \eqref{convolution-estimate-2}. An integration by parts gives
    \begin{align*}
        u_\nu(z)=&\int_\C \log|z-\zeta|\chi_\eta(\zeta)\,\ddc u(\zeta)=
    \int_\C\delta_z\chi_\eta(\zeta)u(\zeta)+\\
    & \int_\C\Big(\log|z-\zeta|\ddc\chi_\eta(\zeta)
        +\d\log|z-\zeta|\wedge\dc\chi_\eta(\zeta)+\d\chi_\eta(\zeta)\wedge\dc\log|z-\zeta|\Big)u(\zeta),
    \end{align*}
    from which it follows that, for every $z\in K_{2\eta}$, $u_\nu(z)-u(z)$ is equal to
    \begin{align} \label{uminusutilde}
       \int_{K_{2\eta}^c\cap K_{3\eta}}\Big(&\log|z-\zeta|\ddc\chi_\eta(\zeta)
        +\d\log|z-\zeta|\wedge\dc\chi_\eta(\zeta)+\d\chi_\eta(\zeta)\wedge\dc\log|z-\zeta|\Big)u(\zeta).
    \end{align}
    Differentiating \eqref{uminusutilde} under the integral sign, we get $\|u-u_\nu\|_{\Cc^1(K_\eta)} \lesssim_{K,\eta} 1$. It follows that
    $$\int_K \big|(u-u_\nu)(z-w)-(u-u_\nu)(z) \big|\,\diff\Leb(z) \lesssim_{K,\eta}\varepsilon\quad\text{for every}\quad |w|\le\varepsilon.$$
    Writing $u=u_\nu+(u_\nu-u)$, we then obtain \eqref{convestquasifinal}.

    \medskip
    
    \noindent
{\bf Conclusion}. Let $u$ be as in the statement. Assume without loss of generality that $\|u\|_{L^\infty(U)}=1$. Take $w\in\C^k$ with $|w|\le\varepsilon$. Setting $z=(\hat{z},z_k)$ with $\hat{z}\in\C^{k-1}$ and $z_k\in\C$, taking $R$ sufficiently large (depending on $U$ and $U'$), and assuming without loss of generality that $w$ has the form $w=(0,w_k)$, we have
    \begin{align}
        \int_{U'}& \big|u(z-w)-u(z)\big|\,\diff\Leb(z) \nonumber\\
        &\le \int_{\D_R^{k-1}}\left(\int_{(\{\hat{z}\}\times\C)\cap U'} \big|u(\hat{z},z_k-w_k)-u(\hat{z},z_k)\big|\,\diff\Leb(z_k)\right)\,\diff\Leb(\hat{z})\nonumber\\
        &\lesssim_{U,U'} -\int_{\D_R^{k-1}} \varepsilon\log{\varepsilon}\,\diff\Leb(\hat{z}) \lesssim_{U,U'}-\varepsilon\log{\varepsilon},\label{eloge-estimate}
    \end{align}
    where in the second inequality we used \eqref{convestquasifinal}. The assertion follows from \eqref{eloge-estimate} and the definition of $u_\varepsilon$.
    \endproof

We will also need the following regularization result. The third item corrects an inaccuracy in \cite[first inequality in (3.1)]{wu-kahler}, which affects the estimate in \cite[Lemma 3.2]{wu-kahler}. Those estimates should be $\|\phi_\varepsilon-\phi\|_{L^1(\omega^k)}\lesssim -1/\log{\varepsilon}$ and $|g(\varepsilon)-g(0)|\lesssim (-1/\log{\varepsilon})^\alpha$ respectively.

\begin{proposition}\label{convol}
    Let $(X,\omega)$ be a compact K\"ahler manifold and $\varphi$ a bounded quasi-p.s.h.\ function  such that $\ddc \varphi\ge-\omega_0$ for some smooth positive closed $(1,1)$-form $\omega_0$. For every $0<\varepsilon\le1/2$, there exists a smooth function $\varphi_\varepsilon$ with $\varphi_\varepsilon\ge\varphi$ and such that:
    \begin{enumerate}[label=(\roman*)]
        \item $\|\varphi_\varepsilon\|_\infty\lesssim\|\varphi\|_\infty$;
        \item $\|\varphi_\varepsilon\|_{\Cc^2}\lesssim \|\varphi\|_\infty\varepsilon^{-2}$;
        \item $\|\varphi_\varepsilon-\varphi\|_{L^1(\omega^k)}\lesssim-\|\varphi\|_\infty/\log{\varepsilon}$;
        \item $\ddc\varphi_\varepsilon\ge-\omega_0$,
    \end{enumerate}
    where the implicit constants depend only on $(X,\omega)$.
\end{proposition}

\begin{proof}
    We follow the proof of \cite[Theorem 2.1]{din-ma-ngu-ens}, where the authors cover the more restrictive case where $\varphi$ is also H\"older-continuous. Items (i), (ii) and (iv) can be proved in the same way, as the regularity of $\varphi$ is not used in their proofs. Instead of the desired estimate in item (iii), in \cite{din-ma-ngu-ens} a stronger result is obtained, namely
    \begin{equation}\label{dmn-est}
        \|\varphi_\varepsilon-\varphi\|_\infty\lesssim\|\varphi\|_{\Cc^\alpha} \varepsilon^\alpha  \quad\text{for some}\quad 0<\alpha\le1,
    \end{equation}
    using the H\"older-continuity of $\varphi$. We cannot obtain the same estimate since we assume $\varphi$ only to be bounded. Inequality \eqref{dmn-est} follows from \cite[first inequality in (2.3)]{din-ma-ngu-ens}:
    \begin{equation}\label{dmn-est2}
        \|u_\delta-u\|_{L^\infty(U')}\lesssim\|u\|_{\Cc^\alpha(U)} \delta^\alpha \quad\text{for }U\text{ open and }U'\Subset U,
    \end{equation}
    where $u$ is a p.s.h.\ function defined on a chart that differs from $\varphi$ by a smooth function and $u_\delta$ is the convolution defined as in \eqref{convolution} with $\delta$ instead of $\varepsilon$. Instead of \eqref{dmn-est2}, we use Lemma \ref{lem-u-vep}. This gives local regularized approximations $u_\delta$ satisfying
    $$\|u_\delta-u\|_{L^1(U',\Leb)}\lesssim_{U,U'}-\|u\|_{L^\infty(U)} \delta\log{\delta} \quad\text{for }U\text{ open and }U'\Subset U.$$
    We then need to glue them using charts. In order to be sure that the gluing works, as done in \cite[Theorem 2.1]{din-ma-ngu-ens}, we apply point c) of \cite[Chapter I, Lemma 5.18]{demailly:agbook}. To do this, we need a uniform estimate for the difference of regularizations done using different charts. We take $U$ compactly contained in a chart $W$, let $F:W\rightarrow W'$ be a biholomorphism ($F$ being a change of charts), and put $u^F_\delta:=(u\circ F^{-1})_\delta\circ F$. In order to conclude, one needs to show that
    \begin{equation} \label{bkestimate}
        \|u_\delta^F-u_\delta\|_{L^\infty(U)} \lesssim_{W,W',U} -\|u\|_{L^\infty(W)}/\log{\delta}.
    \end{equation}
    Since we assume that $\varphi$ is bounded, \eqref{bkestimate} follows directly from the proof of \cite[Lemma 4]{blocki-kolodziej}. This completes the proof.
\end{proof}

A positive measure $\nu$ on $X$ is said to be \emph{moderate} if, for every bounded family $\Fc$ of d.s.h.\ functions on $X$, there exist constants $\alpha>0$ and $c>0$ such that 
\begin{equation*}\label{mod}
\nu \big\{z\in X: \,|\psi(z)|>M \big\}\leq ce^{-\alpha M}\qquad\text{for every }M\geq 0\text{ and }\psi\in\Fc,
\end{equation*} 
see \cite{dinh-exponential,dinh-sibony:cime}. Moderate measures are PB. We have the following result, which is proven in the case of $\P^k$ in \cite[Lemma 2.3]{wu-vergamini-24}. The same proof applies in the general case of compact K\"ahler manifolds.

\begin{lemma} \label{tail-split}
Let $\varphi$ be a non-positive d.s.h.\ function on $X$, satisfying $\|\varphi\|_\DSH\le1$ and $\ddc\varphi\ge-\omega$. Let $\nu$ be a moderate measure on $X$. For every $N\ge0$, we can write $\varphi=\varphi^{(N)}_1+\varphi^{(N)}_2$, where $\varphi^{(N)}_1$ is quasi-p.s.h., with:
$$\ddc\varphi_1^{(N)}\ge-\omega,\quad\|\varphi^{(N)}_1\|_\infty\le N, \quad\text{and}\quad\|\varphi^{(N)}_2\|_{L^q(\nu)}\le C_qe^{-\alpha N/q}$$ for every $q\ge1$, where $\alpha>0$ is a constant independent of $\varphi$ and $q$, and $C_q>0$ is a constant independent of $\varphi$.
\end{lemma}

\subsection{Super-potentials of currents on compact K\"ahler manifolds}

Denote by $\Dc_q$ the real space generated by all positive closed $(q,q)$-currents on $X$. Define a norm $\|\cdot\|_*$ on $\Dc_q$ by
$$\|\Omega\|_*:=\min\big\{\|\Omega^+\|+\|\Omega^-\|\},$$
where the minimum is taken over all positive closed currents $\Omega^\pm$ such that $\Omega=\Omega^+-\Omega^-$. Observe that $\|\Omega^\pm\|$ only depend on the cohomology classes of $\Omega^\pm$ in $H^{q,q}(X,\mathbb{R})$.

\smallskip

We will consider the following topology on $\Dc_q$: given a sequence of currents $(S_n)_{n\ge0}$ and a current $S$, we say that the $S_n$'s converge to $S$ if they converge in the sense of currents and $\|S_n\|_*$ is uniformly bounded. We call this topology the \emph{$*$-topology}. By \cite{DS04}, smooth forms are dense in $\Dc_q$ with respect to the $*$-topology. They are also dense in the space $\Dc_q^0$ given by those currents $S\in\Dc_q$ which are exact, i.e., whose cohomology class $\{S\}$ in $H^{q,q}(X,\R)$ is $0$.

For every $0<l<+\infty$, denote by $\|\cdot\|_{\Cc^l}$  the standard $\Cc^l$ norm on the space of differential forms. We consider a norm $\|\cdot\|_{\Cc^{-l}}$ defined by
$$\|S\|_{\Cc^{-l}}:=\sup_{\|\Phi\|_{\Cc^l}\le1}|\langle S,\Phi\rangle|,$$
where the supremum is on smooth $(k-q,k-q)$-forms $\Phi$ on $X$. Observe that, by interpolation \cite{triebel}, for every $0<l<l'<+\infty$ and $m>0$ there exists a positive constant $c_{l,l',m}$ such that
\begin{equation}\label{ineq-dist}
    \|S\|_{\Cc^{-l'}}\le\|S\|_{\Cc^{-l}}\le c_{l,l',m}\|S\|_{\Cc^{-l'}}^{l/l'}\quad\text{for all }S\text{ such that }\|S\|_*\le m.
\end{equation}

\medskip

Following \cite{dinh-sibony:acta,ds-superpot-kahler}, we now recall the definition of the \emph{super-potential} of a current $S\in\Dc_q$. Fix a basis $\{\alpha\}:=\big\{\{\alpha_1\},\dots,\{\alpha_t\}\big\}$ of $H^{q,q}(X,\R)$. We can take all the $\alpha_j$'s to be smooth. For any $R\in\Dc_{k-q+1}^0$, there exists a real $(k-q,k-q)$-current $U_R$ such that $\ddc U_R=R$. We call $U_R$ a \emph{potential} of $R$. After adding some smooth real closed form to $U_R$, we can assume that $U_R$ is \emph{$\alpha$-normalized}, i.e., that $\langle U_R,\alpha_j\rangle=0$ of all $1\le j\le t$. We can choose $U_R$ smooth if $R$ is smooth. The \emph{$\alpha$-normalized super-potential} $\Uc_S$ of $S$ is the linear functional on the smooth forms in $\Dc_{k-q+1}^0$ which is defined by
$$\Uc_S(R):=\langle S,U_R\rangle.$$
Note that $\Uc_S(R)$ does not depend on the choice of $U_R$.

\smallskip

We say that $S$ has a \emph{continuous super-potential} if $\Uc_S$ can be extended continuously to a linear functional on all of $\Dc_{k-q+1}^0$ with respect to the $*$-topology. If $S\in\Dc_q^0$, then $\Uc_S$ does not depend on the choice of $\alpha$. If $S$ is smooth, then it has a continuous super-potential and for every $R\in\Dc_{k-q+1}^0$ we have $\Uc_S(R)=\Uc_R(S)$, where $\Uc_R$ is the super-potential of $R$. The equality still holds if we only assume that $S$ has a continuous super-potential, see \cite{ds-superpot-kahler}.

\begin{definition} \label{holder-superpot}
    Take $S\in\Dc_q$. For $l>0,0<\lambda\le1$, and $M>0$, we say that a super-potential $\Uc_S$ of $S$ is \emph{$(l,\lambda,M)$-H\"older-continuous} if it is continuous and we have
    $$|\Uc_S(R)|\le M\|R\|_{\Cc^{-l}}^\lambda\quad\text{for every }R\in\Dc_{k-q+1}^0\text{ with }\|R\|_*\le1.$$
\end{definition}
If $S$ is such that $\Uc_S$ is $(l,\lambda,M)$-H\"older-continuous, \eqref{ineq-dist} implies that $\Uc_S$ is also $(l' , \lambda', M')$-H\"older-continuous for every $l'>0$ and some constants $\lambda'$ and $M'$ which depend on $\lambda,M,l',l$, but are independent of $S$. Deﬁnition \ref{holder-superpot} does not depend on the normalization of the super-potential.

\section{Mixing for d.s.h.\ functions} \label{sec-mixing}

In this section, we are going to prove Theorem \ref{thm-mixing}. We follow the general strategy of \cite{wu-vergamini-24}, but we cannot use results about p.s.h.\ functions in the K\"ahler setting. We use instead the techniques from Section \ref{defs}.

\subsection{Mixing for bounded quasi-p.s.h.\ functions} \label{mixing-quasipsh}

Recall that $f$ is a holomorphic automorphism of $X$ with simple action on cohomology, we denote by $d_p$ and $\delta_0$ its main dynamical degree and its auxiliary dynamical degree, by $T_+$ and $T_-$ the Green currents of $f$ and $f^{-1}$ respectively, and by $\mu=T_+\wedge T_-$ the equilibrium measure of $f$. From \cite{DS05,ds-superpot-kahler} we have $f^*(T_+)=d_pT_+$ and $f_*(T_-)=d_pT_-$. Moreover, for every positive closed $(p,p)$-current (respectively, $(k-p,k-p)$-current) $S$ of mass $1$, we have that $d_p^{-n}(f^n)^*(S)$ converges to $T_+$ (respectively, $d_{k-p}^{-n}f^n_*(S)$ converges to $T_-$). We also have that $T_+$ (respectively, $T_-$) is the unique positive closed current in the class $\{T_+\}$ (respectively, $\{T_-\}$).

\medskip

We start establishing a weaker version of Theorem \ref{thm-mixing}  for bounded quasi-p.s.h.\ functions.

\begin{proposition} \label{prop-mixing-bounded}
There exists $\delta_0<\delta<d_p$ such that for every $\kappa\in\mathbb{N}^*$ there exists a constant $C_\kappa>0$ such that, for every   $\kappa+1$ bounded quasi-p.s.h.\ functions $g_0,g_1, \dots,g_\kappa$, and every $0=n_0\leq n_1\leq \ldots\leq n_\kappa$, we have
$$\Big|\int g_0  ( g_1\circ f^{n_1}  ) \cdots (g_\kappa \circ f^{n_\kappa}  )\,\dd\mu -\prod_{j=0}^\kappa \int  g_j  \,\dd\mu\Big|\leq C_\kappa \Big(\frac{\delta}{d_p}\Big)^{\min_{0\leq j\leq \kappa-1}(n_{j+1}-n_j)/2}  \prod_{j=0}^\kappa \norm{g_j}_\qpsh,$$
where we set
$$\|g\|_\qpsh:=\|g\|_\infty+\inf \big\{c\ge0\mid \ddc g\ge -c\omega\big\}\quad\text{for every }g:X\to \R.$$
\end{proposition}

Observe that, by linearity, Proposition \ref{prop-mixing-bounded} also holds if we assume that for every $j$ either $g_j$ or $-g_j$ is quasi-p.s.h.. Observe also that it is already stronger than \cite[Theorem 1.2]{bd-kahler}.

\smallskip

The explicit choice of $\delta$ will be made later. Specifically, we will find $\delta_0<\delta'<d_p$ such that the statement holds for every $\delta'<\delta<d_p$, see \eqref{deltaprime} below.

\medskip

Consider now the K\"ahler manifold $X\times X$ equipped with the K\"ahler form $\tilde{\omega}=\pi_1^*\omega+\pi_2^*\omega$, where $\pi_1,\pi_2$ are the canonical projections of $X\times X$ onto its factors. Define a new automorphism of  $X\times X$ by $$F(z,w):=\big(f(z),f^{-1}(w)\big).$$ Using K\"unneth formula, one can show that the dynamical degree of order $k$ of $F$ is equal to $d_p^2$ (see also \cite[Section 4]{DS10CM}), which is an eigenvalue of multiplicity $1$ of $F^*$, and that all the others dynamical degrees and the eigenvalues of $F^*$ on $H^{k,k}(X\times X,\mathbb{R})$, except for $d_p^2$, are strictly smaller than $d_p\delta_0$. Hence $F$ and $d_p\delta_0$ satisfy the same conditions of $f$ and $\delta_0$.

It is not hard to see that the Green $(k,k)$-currents of $F$ and $F^{-1}$ are $\mathbb T_+:=T_+\otimes T_-$ and $\mathbb T_-:= T_-\otimes T_+$ respectively (see \cite[Section 4.1.8]{federer} for the tensor product of currents) and that they satisfy
$$F^*(\mathbb T_+)=d_p^2\mathbb T_+\quad\text{and}\quad F_*(\mathbb T_-)=d_p^2\mathbb T_-.$$
In particular, they have $(1,\lambda,M)$-H\"older-continuous super-potentials for some $M>0$ and $0<\lambda\le1$, see \cite[Lemma 4.2.5]{ds-superpot-kahler}. Let $\Delta$ denote the diagonal of $X\times X$. Then $[\Delta]$ is a positive closed $(k,k)$-current on $X\times X$.

\smallskip

When proving Proposition \ref{prop-mixing-bounded}, we can assume without loss of generality that $\|g_j\|_\qpsh\le1$ for every $j$, which implies $\|g_j\|_\infty\le1$ and $\ddc g_j\ge-\omega$. On $X\times X$, we define
$$G_0(z,w):=g_0(w)\quad\text{and}\quad G_j(z,w)=g_j(z) \,\text{ for }\, j\ge 1.$$
Notice that the $G_j$'s are quasi-p.s.h.\ for every $j$, and they satisfy $\|G_j\|_\infty=\|g_j\|_\infty$. Since $\ddc g_j\ge-\omega$ for every $j$, we also have that $\ddc G_j\ge-\tilde{\omega}$.

Set $l_0:=0$ and $l_j:=n_j-n_1$ for $j\geq1$, and set $\widetilde G_j:=G_j\circ F^{l_j}$ for every $j$. Define the auxiliary quasi-p.s.h.\ functions $\Phi^\pm$ on $X\times X$ by
\begin{equation}\label{phipm}
\Phi^\pm:=\Phi^\pm_{n_0,\dots,n_\kappa}=\sum_{j=0}^\kappa \Big((\kappa+1)\widetilde{G}_j+\frac{\kappa}{2}\widetilde{G}^2_j\Big)\pm \prod_{j=0}^\kappa \widetilde{G}_j.
\end{equation}
As in \cite{wu-vergamini-24}, these two functions will play a very important role in the proof of Proposition \ref{prop-mixing-bounded}. We have the following estimate for $\ddc\Phi^\pm$.

\begin{lemma} \label{ddc-positive}
   We have $$\ddc \Phi^\pm \gtrsim_\kappa -\displaystyle\sum_{j=0}^\kappa(F^{l_j})^*\tilde{\omega}=:-\omega_0,$$
   where the implicit constant is independent of $n_1,\dots,n_\kappa,g_0,\dots,g_\kappa$.
\end{lemma}

\begin{proof}
    Remember that the inequality $i\del(g\pm h)\wedge\dbar(g\pm h)\ge0$, which is valid for all bounded d.s.h.\ functions $g$ and $h$, implies
    \begin{equation} \label{rearrangement}
    \pm(i\del g\wedge\dbar h+i\del h\wedge\dbar g)\ge-(i\del g\wedge\dbar g+i\del h\wedge\dbar h).
    \end{equation}
    From \eqref{rearrangement}, it follows that we have
    \begin{align}
        i\ddbar\Phi^\pm=&\sum_{j=0}^\kappa i\ddbar\widetilde{G}_j\Big(\kappa+1+\kappa\widetilde{G}_j\pm\prod_{s \not=j}\widetilde{G}_s\Big)+\kappa\sum_{j=0}^\kappa i\del\widetilde{G}_j\wedge\dbar\widetilde{G}_j
        \pm \sum_{j\not=s}\Big(i\del \widetilde{G}_j\wedge\dbar \widetilde{G}_s\prod_{t\neq j,s}\widetilde{G}_t\Big)  \nonumber\\
        & \ge \sum_{j=0}^\kappa i\ddbar\widetilde{G}_j\Big(\kappa+1+\kappa\widetilde{G}_j\pm\prod_{s\not=j}\widetilde{G}_s\Big)+\sum_{j=0}^\kappa i\del \widetilde{G}_j\wedge\dbar \widetilde{G}_j\bigg(\kappa-\sum_{s\not=j}  \Big(\prod_{t\neq j,s}|\widetilde{G}_t| \Big)\bigg).\label{ddc1}
    \end{align}
    Using the fact that $\|G_j\|_\infty=\|g_j\|_\infty\le1$ and $\ddc G_j\ge-\tilde{\omega}$ for every $j$, we get
    \begin{align}
        \sum_{j=0}^\kappa &i\ddbar\widetilde{G}_j\Big(\kappa+1+\kappa\widetilde{G}_j\pm\prod_{s\not=j}\widetilde{G}_s\Big)+\sum_{j=0}^\kappa i\del \widetilde{G}_j\wedge\dbar \widetilde{G}_j\bigg(\kappa-\sum_{s\not=j}  \Big(\prod_{t\neq j,s}|\widetilde{G}_t| \Big)\bigg)\nonumber\\
        &\ge \sum_{j=0}^\kappa i\ddbar\widetilde{G}_j\Big(\kappa+1+\kappa\widetilde{G}_j\pm\prod_{s\not=j}\widetilde{G}_s\Big) \gtrsim_\kappa -\sum_{j=0}^\kappa(F^{l_j})^*\tilde{\omega}.\label{ddc2}
    \end{align}
    The assertion follows from \eqref{ddc1} and \eqref{ddc2}.
\end{proof}

We deduce from the above lemma the following result, which is obtained applying Proposition \ref{convol} to the functions $\Phi^\pm$.

\begin{corollary} \label{cor-epsilon}
For every $0<\varepsilon\le1/2$, there are two regularized functions $\Phi^\pm_\varepsilon$ with $\Phi^\pm_\varepsilon\ge\Phi^\pm$ and such that:
\begin{enumerate}[label=(\roman*)]
    \item $\|\Phi^\pm_\varepsilon\|_\infty\lesssim_\kappa 1$;
    \item $\|\Phi^\pm_\varepsilon\|_{\Cc^2} \lesssim_\kappa \varepsilon^{-2}$;
    \item $\|\Phi^\pm_\varepsilon-\Phi^\pm\|_{L^1(\tilde{\omega}^{2k})}\lesssim_\kappa -1/\log{\varepsilon}$;
    \item $\ddc\Phi^\pm_\varepsilon\gtrsim_\kappa-\omega_0$.
\end{enumerate}
\end{corollary}

We have the following lemma about the functions $\Phi^\pm_\varepsilon$. As in \cite[Lemma 3.4]{wu-vergamini-24}, a delicate point of this estimate is the independence of the $n_j$'s. Furthermore, here we also need the independence of $\varepsilon$, which is a consequence of the above corollary. Moreover, we will see here the crucial role of the assumption on the simple action on cohomology of $f$. This is implicitly used in \cite{wu-vergamini-24} as every H\'enon-Sibony map satisfies this condition.

\begin{lemma} \label{lem-star-bound}
For every $0<\varepsilon\le1/2$, we have $\norm{\ddc\Phi^\pm_\varepsilon\wedge\mathbb T_+}_* \leq c_\kappa$ for some constant $c_\kappa>0$ which is independent of $n_1,\dots, n_\kappa$, $g_0,\dots,g_\kappa$, and $\varepsilon$.
\end{lemma}

\begin{proof}
We deduce from Corollary \ref{cor-epsilon} (iv) that we have
\begin{equation} \label{*estimate1}
\ddc\Phi^\pm_\varepsilon \wedge \mathbb T_+ \gtrsim_\kappa -\omega_0\wedge \mathbb T_+= -\sum_{j=0}^\kappa(F^{l_j})^*\tilde{\omega}\wedge \mathbb T_+=:-\Omega_0.
\end{equation}
 We will show that, for every $j$, the mass of $(F^{l_j})^*\tilde{\omega} \wedge \mathbb T_+$ is bounded independently of $n_1,\dots, n_\kappa$. Using that $F^*(\mathbb T_+)=d_p^2\mathbb T_+$, we have 
\begin{equation} \label{*estimate2}
(F^{l_j})^*\tilde{\omega} \wedge \mathbb T_+  = d_p^{-2l_j}(F^{l_j})^*( \tilde{\omega} \wedge \mathbb T_+).
\end{equation}

Since the mass of a positive closed current can be computed cohomologically and $d_{k+1}(F)\le d_p\delta_0$, for every current $R$ in $\Dc_{k+1}(X\times X)$ we have $\|(F^n)^*(R)\|_* \lesssim (d_p\delta_0)^n \|R\|_*$. For every $j$, it follows that we have
\begin{equation}\label{*estimate3}
    \big\| (F^{l_j})^*( \tilde{\omega} \wedge \mathbb T_+) \big\| \lesssim (d_p\delta_0)^{l_j}\|\tilde{\omega} \wedge \mathbb T_+ \|_* \lesssim (d_p\delta_0)^{l_j}.
\end{equation}
We can write $\ddc\Phi^\pm_\varepsilon\wedge\mathbb T_+$ as $(\ddc\Phi^\pm_\varepsilon\wedge\mathbb T_++\tilde{c}_\kappa\Omega_0)-\tilde{c}_\kappa\Omega_0$, which is the difference of two positive currents, where $\tilde{c}_\kappa$ is the implicit constant in \eqref{*estimate1}. Since $\ddc\Phi^\pm_\varepsilon\wedge\mathbb T_+$ is exact, the mass of $\ddc\Phi^\pm_\varepsilon\wedge\mathbb T_++\tilde{c}_\kappa\Omega_0$ is equal to $\|\tilde{c}_\kappa\Omega_0\|$. Hence, combining \eqref{*estimate2} and \eqref{*estimate3} and using the definitions of $\Omega_0$ and $\|\cdot\|_*$ gives the statement.
\end{proof}

From \cite[Proposition 3.4.2]{ds-superpot-kahler}, we know that $\mathbb T_+\wedge \mathbb T_-$ has a $(2,\lambda_0,M)$-H\"older-continuous super-potential for some $0<\lambda_0\le1$ and $M>0$. Set
\begin{equation}\label{deltaprime}
    \delta':=d_p^{\frac{1}{1+\lambda_0}}\delta_0^{\frac{\lambda_0}{1+\lambda_0}},
\end{equation}
and observe that $\delta_0<\delta'<d_p$. We will prove Proposition \ref{mixing-quasipsh} and Theorem \ref{thm-mixing} for every $\delta'<\delta<d_p$. This is equivalent to ask that
$$\tilde{\delta}:=d_p^{1/\lambda_0}\delta_0/\delta^{1/\lambda_0}<\delta.$$

\begin{remark}\label{lambdazero}
    One can actually prove that $\lambda_0\ge\displaystyle\frac{1}{8}\left(\frac{\log(d_p/\delta'')}{\log(d_p/\delta'')+\log{A}}\right)^2$, where $A=\|F^*\|_{\Cc^1}$ and $\delta''$ is any real number between $\delta_0$ and $d_p$, see for instance \cite[Lemma 4.2.5]{ds-superpot-kahler}. Hence, $\delta$ depends only on the dynamical degrees and the Lipschitz constant of $f$. In particular, it can be taken to depend continuously on $f$.
\end{remark}

\smallskip

Let now $S$ be a fixed positive closed $(k,k)$-current of mass $1$ on $X\times X$. We will need the following estimate, see also \cite[Proposition 3.3]{wu-kahler}.

\begin{proposition} \label{key-ineq}
     Let $S$ be a positive closed $(k,k)$-current such that $S_n:=d_p^{-2n}F^n_*(S)$ converges to $\mathbb T_-$. There exists a constant $c_\kappa>0$, independent of $\Phi^\pm$, such that for all $n$ we have
    $$\langle S_n\wedge \mathbb T_+,\Phi^\pm\rangle-\langle \mathbb T_-\wedge \mathbb T_+,\Phi^\pm\rangle \le c_\kappa(\delta/d_p)^n.$$
\end{proposition}

In order to prove Proposition \ref{key-ineq}, we follow the proof of \cite[Proposition 3.3]{wu-kahler}. Every step applies, but we have to correct the use of the estimate in \cite[Lemma 3.2]{wu-kahler}, see the comment before Proposition \ref{convol}. That estimate, applied to $X\times X$ and $\Phi^\pm$, says that
    \begin{equation} \label{hao-k-1}
        \big|\Uc_{\mathbb T_+\wedge\mathbb T_-}(\ddc\Phi^\pm_\varepsilon)-\Uc_{\mathbb T_+\wedge\mathbb T_-}(\ddc\Phi^\pm)\big| \lesssim_\kappa \varepsilon^{\lambda_0}.
    \end{equation}
    Inequality \eqref{hao-k-1} is a consequence of \cite[first inequality in (3.1)]{wu-kahler}, which in the case of $X\times X$ and $\Phi^\pm$ becomes
    \begin{equation} \label{hao-k-2}
        \|\Phi^\pm_\varepsilon-\Phi^\pm\|_{L^1(\tilde{\omega}^{2k})}\lesssim_\kappa \varepsilon.
    \end{equation}
    On the other hand, we have seen in Corollary \ref{cor-epsilon} (iii) that \eqref{hao-k-2} holds with $-1/\log{\varepsilon}$ instead of $\varepsilon$ in the right hand side, see \eqref{citeabove} below.

\begin{proof}[Proof of Proposition \ref{key-ineq}.]
    From Corollary \ref{cor-epsilon} (i) and Lemma \ref{lem-star-bound}, we have $\|\Phi^\pm_\varepsilon\|_\infty\lesssim_\kappa 1$ and $\norm{\ddc\Phi^\pm_\varepsilon\wedge\mathbb T_+}_*\lesssim_\kappa 1$ for every $0<\varepsilon\le1/2$. Hence, up to rescaling, we can assume without loss of generality that we have $\|\Phi^\pm_\varepsilon\|_\infty \le 1$ and $\norm{\ddc\Phi^\pm_\varepsilon\wedge\mathbb T_+}_*\le1$. The $(2,\lambda,M\varepsilon^{-2})$-H\"older-continuity of the super-potentials of $\ddc\Phi^\pm_\varepsilon\wedge\mathbb T_+$, for some $M>0$ and $0<\lambda\le1$, follows from Corollary \ref{cor-epsilon} (ii).

    \smallskip

    From the fact that $\Phi^\pm_\varepsilon\ge\Phi^\pm$ and a direct computation, we get
    \begin{align} \label{many-potentials}
    \langle S_n\wedge \mathbb T_+,&\Phi^\pm\rangle-\langle \mathbb T_-\wedge \mathbb T_+,\Phi^\pm\rangle\\
    &\le\langle S_n\wedge \mathbb T_+,\Phi^\pm_\varepsilon\rangle-\langle \mathbb T_-\wedge \mathbb T_+,\Phi^\pm\rangle\nonumber\\
    &=\langle S_n\wedge \mathbb T_+,\Phi^\pm_\varepsilon\rangle-\langle \mathbb T_-\wedge \mathbb T_+,\Phi^\pm_\varepsilon\rangle+\langle \mathbb T_-\wedge \mathbb T_+,\Phi^\pm_\varepsilon\rangle-\langle \mathbb T_-\wedge \mathbb T_+,\Phi^\pm\rangle\nonumber\\
    &=\begin{aligned}[t]&\Uc_{S_n}(\ddc\Phi^\pm_\varepsilon\wedge\mathbb T_+)-\Uc_{\mathbb T_-}(\ddc\Phi^\pm_\varepsilon\wedge\mathbb T_+)+\langle S_n,K^\pm_\varepsilon\rangle-\langle \mathbb T_-,K^\pm_\varepsilon\rangle\nonumber\\
    &+\Uc_{\mathbb T_+\wedge\mathbb T_-}(\ddc\Phi^\pm_\varepsilon)+\langle\tilde{\omega}^{2k},\Phi^\pm_\varepsilon\rangle-\Uc_{\mathbb T_+\wedge\mathbb T_-}(\ddc\Phi^\pm)-\langle\tilde{\omega}^{2k},\Phi^\pm\rangle,\end{aligned}
    \end{align}
    where $K_\varepsilon$ is a smooth closed $(k,k)$-form such that $\Phi^\pm_\varepsilon\mathbb T_+-K_\varepsilon$ is a normalized super-potential of $\ddc\Phi^\pm_\varepsilon\wedge\mathbb T_+$. From Corollary \ref{cor-epsilon} (iii) we have that
    \begin{equation}\label{citeabove}
    \|\Phi^\pm_\varepsilon-\Phi^\pm\|_{L^1(\tilde{\omega}^{2k})}\lesssim_\kappa -1/\log{\varepsilon}.
    \end{equation}
    From \eqref{citeabove} we deduce
    \begin{equation} \label{logepsilonnotalpha}
    \big|\langle\tilde{\omega}^{2k},\Phi^\pm_\varepsilon\rangle-\langle\tilde{\omega}^{2k},\Phi^\pm\rangle \big| \lesssim -1/\log{\varepsilon}
    \end{equation}
    and, using the $(2,\lambda_0,M)$-H\"older-continuity of $\Uc_{\mathbb T_+\wedge\mathbb T_-}$,
    \begin{equation} \label{logepsilonalpha}
    \big| \Uc_{\mathbb T_+\wedge\mathbb T_-}(\ddc\Phi^\pm_\varepsilon)-\Uc_{\mathbb T_+\wedge\mathbb T_-}(\ddc\Phi^\pm) \big| \lesssim (-1/\log{\varepsilon})^{\lambda_0}.
    \end{equation}
    From \cite[Proposition 2.4]{wu-kahler} and \cite[Lemma 3.1]{wu-kahler} we have
    \begin{equation}\label{doubletrouble1}
    \big|\Uc_{S_n}(\ddc\Phi^\pm_\varepsilon\wedge\mathbb T_+)-\Uc_{\mathbb T_-}(\ddc\Phi^\pm_\varepsilon\wedge\mathbb T_+)\big|\lesssim -(\delta_0/d_p)^n\log\varepsilon
    \end{equation}
    and
    \begin{equation}\label{doubletrouble2}
    \big|\langle S_n,K^\pm_\varepsilon\rangle-\langle \mathbb T_-,K^\pm_\varepsilon\rangle\big|\lesssim (\delta_0/d_p)^n,
    \end{equation}
    respectively. Combining \eqref{many-potentials}, \eqref{logepsilonnotalpha}, \eqref{logepsilonalpha}, \eqref{doubletrouble1} and \eqref{doubletrouble2}, we get
    $$\langle S_n\wedge \mathbb T_+,\Phi^\pm_\varepsilon\rangle-\langle \mathbb T_-\wedge \mathbb T_+,\Phi^\pm\rangle \lesssim -(\delta_0/d_p)^n\log{\varepsilon}+(\delta_0/d_p)^n+(-1/\log{\varepsilon})^{\lambda_0}-1/\log{\varepsilon}.$$
    We just need to prove the statement for $n$ sufficiently large. It then suffices to choose $\varepsilon :=e^{-(d_p/\delta)^{n/\lambda_0}}$. We get
    $$\langle S_n\wedge \mathbb T_+,\Phi^\pm\rangle-\langle \mathbb T_-\wedge \mathbb T_+,\Phi^\pm\rangle \lesssim (\tilde{\delta}/d_p)^n+(\delta_0/d_p)^n+(\delta/d_p)^n+(\delta/d_p)^{n/\lambda_0}\lesssim (\delta/d_p)^n.$$
    The proof is complete.
\end{proof}

We can now prove Proposition \ref{prop-mixing-bounded}. Using the invariance of $\mu$, the desired inequality does not change if we replace $n_j$ by $n_j-1$ for $1\leq j\leq \kappa$ and $g_0$ by $g_0\circ f^{-1}$. Therefore, it is enough to assume that $n_1$ is even. We have the following lemma.

\begin{lemma}\label{lem-Psi}
There is a constant $c_\kappa>0$, independent of $n_1,\dots, n_\kappa$ and $g_0,\dots,g_\kappa$, such that 
$$ \Big| \int \prod_{j=0}^\kappa (g_j\circ f^{n_j})\,\dd \mu -\int g_0 \,\dd \mu \int \prod_{j=1}^\kappa (g_j\circ f^{n_j-n_1})\,\dd \mu \Big|\leq c_\kappa \Big(\frac{\delta}{d_p}\Big)^{n_1 /2}.    $$
\end{lemma}

\begin{proof}
Put $\Psi:= g_1 (g_2 \circ f^{n_2-n_1})\cdots (g_\kappa \circ f^{n_\kappa-n_1})$. 
We are going to prove that we have
\begin{equation} \label{ineq-mod}
 \Big|\int g_0 ( \Psi\circ f^{n_1})\,\dd \mu -\int g_0 \,\dd \mu \int \Psi \,\dd \mu\Big| \leq c_\kappa \Big(\frac{\delta}{d_p}\Big)^{n_1/2}  \end{equation}
for some $c_\kappa>0$ independent of $n_1,\dots, n_\kappa$ and $g_0,\dots,g_\kappa$. This gives the desired result. We will make use of the functions $\Phi^\pm$ defined in \eqref{phipm}.

\smallskip

Using the invariance of $\mu$ and the definitions of $\Psi$ and $\Phi^\pm$, a direct computation (see for instance \cite[Lemma 3.5]{wu-vergamini-24}) gives
\begin{align*}
\pm\int g_0 (\Psi \circ f^{n_1})\,\dd\mu+\int \Big( (\kappa+1)\sum_{j=0}^\kappa g_j+\frac{\kappa}{2}\sum_{j=0}^\kappa g^2_j \Big)\,\dd\mu = \big\lp \mathbb T_+ \wedge [\Delta]  \,,\, (F^{n_1/2})^* \Phi^\pm \big\rp.
\end{align*}

From the fact that $F^*(\mathbb T_+)=d_p^2\mathbb T_+$, it follows that we have
$$\big\lp \mathbb T_+ \wedge [\Delta]  \,,\, (F^{n_1/2})^* \Phi^\pm \big\rp = \big\lp (F^{n_1/2})_*(\mathbb T_+ \wedge [\Delta])  \,,\,  \Phi^\pm \big\rp  = \big\lp  d_p^{-n_1}(F^{n_1/2})_*[\Delta]\wedge\mathbb T_+  \,,\,   \Phi^\pm \big\rp.$$
Therefore, we have
\begin{equation} \label{split1}
    \pm\int g_0 (\Psi \circ f^{n_1})\,\dd\mu+\int \Big( (\kappa+1)\sum_{j=0}^\kappa g_j+\frac{\kappa}{2}\sum_{j=0}^\kappa g^2_j \Big)\,\dd\mu = \big\lp  d_p^{-n_1}(F^{n_1/2})_*[\Delta]\wedge\mathbb T_+  \,,\,   \Phi^\pm \big\rp.
\end{equation}

Since $\mu\otimes \mu= \mathbb T_+ \wedge \mathbb T_-=\mathbb T_-\wedge\mathbb T_+$, and using also the invariance of $\mu$, we get
\begin{equation} \label{split4}
\int \Big( (\kappa+1)\sum_{j=0}^\kappa g_j+\frac{\kappa}{2}\sum_{j=0}^\kappa g^2_j \Big)\,\dd\mu\pm\lp \mu, g_0 \rp\lp \mu, \Psi  \rp=   \lp   \mu\otimes \mu , \Phi^\pm \rp =\big\lp  \mathbb T_-\wedge\mathbb T_+, \Phi^\pm\big\rp.
\end{equation}

Subtracting \eqref{split4} from \eqref{split1} and applying Proposition \ref{key-ineq} with $S=[\Delta]$, we get \eqref{ineq-mod}. This concludes the proof of the lemma.
\end{proof}

\begin{proof}[End of the proof of Proposition \ref{prop-mixing-bounded}]
We proceed by induction. The base case $\kappa=1$ is given by Lemma \ref{lem-Psi}. Suppose that the statement holds for $\kappa-1$ observables. We need to prove that it holds for $\kappa$, i.e., that we have
$$\Big|\int \prod_{j=0}^\kappa (g_j\circ f^{n_j})\,\dd\mu -\prod_{j=0}^\kappa \int  g_j  \,\dd\mu\Big|\lesssim \Big(\frac{\delta}{d_p}\Big)^{\min_{0\leq j\leq \kappa-1}(n_{j+1}-n_j)/2}.$$
Recall that we can assume that $\|g_j\|_\qpsh\leq 1$ for every $j\ge1$. Again by Lemma \ref{lem-Psi}, it is enough to show that we have
$$ \Big| \int g_0 \,\dd \mu \int \prod_{j=1}^\kappa (g_j\circ f^{n_j-n_1})\,\dd \mu-\prod_{j=0}^\kappa \int  g_j  \,\dd\mu\Big|\lesssim \Big(\frac{\delta}{d_p}\Big)^{-\min_{1\leq j\leq \kappa-1}(n_{j+1}-n_j)/2}.$$
This follows from the inductive assumption. The proof is complete.
\end{proof}

\subsection{Mixing for all d.s.h.\ functions} \label{mixing-dsh}

We can now deduce our main theorem from Proposition \ref{mixing-quasipsh}. As, from now on, the arguments are the same as those in \cite[Theorem 1.2]{wu-vergamini-24}, we will only give a sketch of the proof.

\begin{proof}[Proof of Theorem \ref{thm-mixing}]
Up to rescaling, we can assume without loss of generality that $\|\varphi_j\|_\DSH\le1$ for every $j$. Applying Lemma \ref{dsh-split-psh}, and by linearity, we may also assume that we have
$$\varphi_j\le0,\qquad\|\varphi_j\|_\DSH \le 1,\qquad\text{and}\qquad\ddc\varphi_j\ge-\omega\qquad\text{for every }j.$$

\smallskip

Using Lemma \ref{tail-split}, we can write $\varphi_j=\varphi^{(N)}_{j,1}+\varphi^{(N)}_{j,2}$, where we choose $N$ as
\begin{equation} \label{choiceofN}
    N:=\lfloor(2\alpha)^{-1}\min_{0\leq j\leq \kappa-1}(n_{j+1}-n_j)\log{(d_p/\delta)}\rfloor-1,
\end{equation}
or $N=0$ if the expression in \eqref{choiceofN} is negative. Since $N$ is fixed,  we will omit its dependence and write $\varphi_{j,1}^{(N)}=\varphi_{j,1}$ and $\varphi_{j,2}^{(N)}=\varphi_{j,2}$.

\smallskip

Indexing all the possible choices of the $v_j$'s indexes in the $\varphi_{j,v_j}$'s with $\bv:=(v_0,v_1,\dots,v_\kappa)\in\{1,2\}^{\kappa +1}$, as in \cite[Section 3.2]{wu-vergamini-24} we have
\begin{align*}
    \Big|\int \Big(\prod_{j=0}^\kappa \varphi_j\circ f^{n_j} \Big)\,\dd \mu &-\prod_{j=0}^\kappa  \int \varphi_j \,\dd \mu \Big|\le\Big|\int \Big(\prod_{j=0}^\kappa \varphi_{j,1}\circ f^{n_j} \Big)\,\dd \mu -\prod_{j=0}^\kappa  \int \varphi_{j,1} \,\dd \mu\Big|\\
    & +\sum_{ \bv\not=(1,\dots,1)}\bigg(\Big|\int \Big(\prod_{j=0}^\kappa \varphi_{j,v_j}\circ f^{n_j} \Big)\,\dd \mu\Big|+\Big|\prod_{j=0}^\kappa  \int \varphi_{j,v_j} \,\dd \mu\Big|\bigg).
\end{align*}
To estimate the right hand side of the last expression, we treat two terms separately.

\smallskip
\textbf{Case} $\bv=(1,\dots,1)$. Since all the $\varphi_{j,1}$'s are quasi-p.s.h.\ with $\|\varphi_{j,1}\|_\qpsh \le N+1$ for every $j$, we can apply Proposition \ref{prop-mixing-bounded} to get
\begin{align*}
\Big|\int &\varphi_{0,1}  ( \varphi_{1,1}\circ f^{n_1}  ) \cdots (\varphi_{\kappa,1} \circ f^{n_\kappa}  )\,\dd\mu -\prod_{j=0}^\kappa \int  \varphi_{j,1}  \,\dd\mu\Big|\\
&\leq  C_\kappa \Big(\frac{\delta}{d_p}\Big)^{\min_{0\leq j\leq \kappa-1}(n_{j+1}-n_j)/2}\prod_{j=0}^\kappa \|\varphi_{j,1}\|_\qpsh\leq C_\kappa \Big(\frac{\delta}{d_p}\Big)^{\min_{0\leq j\leq \kappa-1}(n_{j+1}-n_j)/2} (N+1)^{\kappa+1}.
\end{align*}

\textbf{Case} $\bv\not=(1,\dots,1)$. As in \cite[Section 3.2]{wu-vergamini-24}, each of these terms is bounded by $N^{\kappa}e^{-\alpha N}$, up to a multiplicative constant depending only on $\kappa$.

\smallskip

Up to choosing a slightly worse $\delta$, we can conclude the proof as in \cite[Section 3.2]{wu-vergamini-24}.
\end{proof}

\bigskip

%%%%%%%%%%%%%%%%%%


\medskip

\begin{thebibliography}{10}

    \bibitem{bedford1}
    Eric Bedford, Mikhail Lyubich, and John Smillie.
    \newblock Polynomial diffeomorphisms of {$\mathbb C^2$}. {IV}. {T}he measure of maximal entropy and laminar currents.
    \newblock {\em Invent. Math.}, 112(1):77--125, 1993.
    
    \bibitem{bed-1991}
    Eric Bedford and John Smillie.
    \newblock Polynomial diffeomorphisms of {$\mathbb{C}^2$}: currents, equilibrium measure and hyperbolicity.
    \newblock {\em Invent. Math.}, 103(1):69--99, 1991.
    
    \bibitem{bd-kahler}
    Fabrizio Bianchi and Tien-Cuong Dinh.
    \newblock Exponential mixing of all orders and {CLT} for automorphisms of compact {K}{\"a}hler manifolds.
    \newblock {\em {\tt arXiv:2304.13335}}, 2023.
    
    \bibitem{bd-gafa}
    Fabrizio Bianchi and Tien-Cuong Dinh.
    \newblock Equilibrium {S}tates of {E}ndomorphisms of {$\mathbb P^k$}: {S}pectral {S}tability and {L}imit {T}heorems.
    \newblock {\em Geom. Funct. Anal.}, 34(4):1006--1051, 2024.
    
    \bibitem{bian-dinh-sigma}
    Fabrizio Bianchi and Tien-Cuong Dinh.
    \newblock Every complex {H}\'{e}non map is exponentially mixing of all orders and satisfies the {CLT}.
    \newblock {\em Forum Math. Sigma}, 12:Paper No. e4, 2024.
    
    \bibitem{bjo-gor-clt}
    Michael Bj\"{o}rklund and Alexander Gorodnik.
    \newblock Central limit theorems for group actions which are exponentially mixing of all orders.
    \newblock {\em J. Anal. Math.}, 141(2):457--482, 2020.
    
    \bibitem{blocki-kolodziej}
    Zbigniew B\l{}ocki and S\l{}awomir Ko\l{}odziej.
    \newblock On regularization of plurisubharmonic functions on manifolds.
    \newblock {\em Proc. Amer. Math. Soc.}, 135(7):2089--2093, 2007.
    
    \bibitem{C01}
    Serge Cantat.
    \newblock Dynamique des automorphismes des surfaces {$K3$}.
    \newblock {\em Acta Math.}, 187(1):1--57, 2001.
    
    \bibitem{D-TD12}
    Henry De~Th\'elin and Tien-Cuong Dinh.
    \newblock Dynamics of automorphisms on compact {K}\"ahler manifolds.
    \newblock {\em Adv. Math.}, 229(5):2640--2655, 2012.
    
    \bibitem{hen-gab-CLT}
    Henry De~Thélin and Gabriel Vigny.
    \newblock {E}xponential mixing of all orders and {CLT} for generic birational maps of {$\mathbb {P}^k$}.
    \newblock {\em {\tt arXiv:2402.01178}}, 2024.
    
    \bibitem{demailly:agbook}
    Jean-Pierre Demailly.
    \newblock {\em Complex Analytic and Differential Geometry}.
    \newblock \url{http://www-fourier.ujf-grenoble.fr/~demailly/manuscripts/agbook.pdf}.
    
    \bibitem{dinh-decay-henon}
    Tien-Cuong Dinh.
    \newblock Decay of correlations for {H}\'{e}non maps.
    \newblock {\em Acta Math.}, 195:253--264, 2005.
    
    \bibitem{din-ma-ngu-ens}
    Tien-Cuong Dinh, Xiaonan Ma, and Vi\^{e}t-Anh Nguy\^{e}n.
    \newblock Equidistribution speed for {F}ekete points associated with an ample line bundle.
    \newblock {\em Ann. Sci. \'{E}c. Norm. Sup\'{e}r. (4)}, 50(3):545--578, 2017.
    
    \bibitem{dinh-exponential}
    Tien-Cuong Dinh, Vi\^{e}t-Anh Nguy\^{e}n, and Nessim Sibony.
    \newblock Exponential estimates for plurisubharmonic functions and stochastic dynamics.
    \newblock {\em J. Differential Geom.}, 84(3):465--488, 2010.
    
    \bibitem{DS04}
    Tien-Cuong Dinh and Nessim Sibony.
    \newblock Regularization of currents and entropy.
    \newblock {\em Ann. Sci. \'Ecole Norm. Sup. (4)}, 37(6):959--971, 2004.
    
    \bibitem{DS05}
    Tien-Cuong Dinh and Nessim Sibony.
    \newblock Green currents for holomorphic automorphisms of compact {K}\"ahler manifolds.
    \newblock {\em J. Amer. Math. Soc.}, 18(2):291--312, 2005.
    
    \bibitem{din-sib-cmh}
    Tien-Cuong Dinh and Nessim Sibony.
    \newblock Distribution des valeurs de transformations m\'{e}romorphes et applications.
    \newblock {\em Comment. Math. Helv.}, 81(1):221--258, 2006.
    
    \bibitem{dinh-sibony:acta}
    Tien-Cuong Dinh and Nessim Sibony.
    \newblock Super-potentials of positive closed currents, intersection theory and dynamics.
    \newblock {\em Acta Math.}, 203(1):1--82, 2009.
    
    \bibitem{dinh-sibony:cime}
    Tien-Cuong Dinh and Nessim Sibony.
    \newblock Dynamics in several complex variables: endomorphisms of projective spaces and polynomial-like mappings.
    \newblock In {\em Holomorphic dynamical systems}, volume 1998 of {\em Lecture Notes in Math.}, pages 165--294. Springer, Berlin, 2010.
    
    \bibitem{DS10CM}
    Tien-Cuong Dinh and Nessim Sibony.
    \newblock Exponential mixing for automorphisms on compact {K}ähler manifolds.
    \newblock In {\em Dynamical numbers---interplay between dynamical systems and number theory}, volume 532 of {\em Contemp. Math.}, pages 107--114. Amer. Math. Soc., Providence, RI, 2010.
    
    \bibitem{ds-superpot-kahler}
    Tien-Cuong Dinh and Nessim Sibony.
    \newblock Super-potentials for currents on compact {K}\"ahler manifolds and dynamics of automorphisms.
    \newblock {\em J. Algebraic Geom.}, 19(3):473--529, 2010.
    
    \bibitem{federer}
    Herbert Federer.
    \newblock {\em Geometric measure theory}.
    \newblock Classics in Mathematics. Springer, 2014.
    
    \bibitem{forn}
    John~Erik Forn{\ae}ss and Nessim Sibony.
    \newblock Complex {H}\'{e}non mappings in {${\mathbb C}^2$} and {F}atou-{B}ieberbach domains.
    \newblock {\em Duke Math. J.}, 65(2):345--380, 1992.
    
    \bibitem{G90}
    M.~Gromov.
    \newblock Convex sets and {K}\"ahler manifolds.
    \newblock In {\em Advances in differential geometry and topology}, pages 1--38. World Sci. Publ., Teaneck, NJ, 1990.
    
    \bibitem{K79}
    Askold~Georgievich Khovanskii.
    \newblock The geometry of convex polyhedra and algebraic geometry.
    \newblock {\em Uspekhi Mat. Nauk}, 34(4):160--161, 1979.
    
    \bibitem{Te79}
    Bernard Teissier.
    \newblock Du th\'eor\`eme de l'index de {H}odge aux in\'egalit\'es isop\'erim\'etriques.
    \newblock {\em C. R. Acad. Sci. Paris S\'er. A-B}, 288(4):A287--A289, 1979.
    
    \bibitem{triebel}
    Hans Triebel.
    \newblock {\em Interpolation theory, function spaces, differential operators}.
    \newblock Johann Ambrosius Barth, Heidelberg, second edition, 1995.
    
    \bibitem{wu-vergamini-24}
    Marco Vergamini and Hao Wu.
    \newblock {M}ixing and {CLT} for {H}\'enon-{S}ibony maps: plurisubharmonic observables.
    \newblock {\em {\tt arXiv:2407.15418}}, 2024.
    
    \bibitem{vigny-decay}
    Gabriel Vigny.
    \newblock Exponential decay of correlations for generic regular birational maps of {$\mathbb{P}^k$}.
    \newblock {\em Math. Ann.}, 362(3-4):1033--1054, 2015.
    
    \bibitem{wu-kahler}
    Hao Wu.
    \newblock Exponential mixing property for automorphisms of compact {K}\"ahler manifolds.
    \newblock {\em Ark. Mat.}, 59(1):213--227, 2021.
    
    \bibitem{Wu-Ergodic}
    Hao Wu.
    \newblock Exponential mixing property for {H}\'{e}non-{S}ibony maps of {$\mathbb {C}^k$}.
    \newblock {\em Ergodic Theory Dynam. Systems}, 42(12):3818--3830, 2022.
    
    \end{thebibliography}
\end{document}